\documentclass{amsart}
\usepackage{amsmath, amscd, amssymb, amsthm}
\usepackage{bbm}
\usepackage{latexsym}
\usepackage{amsfonts}
\usepackage{graphicx}
\usepackage[all,cmtip]{xy}

\usepackage[
colorlinks=true, linkcolor=black,breaklinks=true,
urlcolor=blue, citecolor=black]{hyperref}%
\usepackage{geometry}
\newtheorem{theorem}{Theorem}
\newtheorem{lemma}{Lemma}

\newcommand{\be}{\begin{equation}}
\newcommand{\ee}{\end{equation}}
\newcommand{\bea}{\begin{eqnarray}}
\newcommand{\eea}{\end{eqnarray}}

\def\XXint#1#2#3{{\setbox0=\hbox{$#1{#2#3}{\int}$ }
\vcenter{\hbox{$#2#3$ }}\kern-.6\wd0}}

\begin{document}

\title[Complex nilmanifolds with constant $H$]{Complex nilmanifolds with constant holomorphic sectional curvature}

\author{Yulu Li}
\address{Yulu Li. School of Mathematical Sciences, Chongqing Normal University, Chongqing 401331, China}
\email{{1320779072@qq.com}}

\author{Fangyang Zheng} \thanks{The research is partially supported by NSFC grant \# 12071050 and Chongqing Normal University.}
\address{Fangyang Zheng. School of Mathematical Sciences, Chongqing Normal University, Chongqing 401331, China}
\email{{20190045@cqnu.edu.cn}}


\begin{abstract}
A well known conjecture in complex geometry states that a compact Hermitian manifold with constant holomorphic sectional curvature must be K\"ahler if the constant is non-zero and must be Chern flat if the constant is zero. The conjecture is confirmed in complex dimension $2$, by the work of Balas-Gauduchon in 1985 (when the constant is zero or negative) and by Apostolov-Davidov-Muskarov in 1996 (when the constant is positive). For higher dimensions, the conjecture is still largely unknown. In this article, we restrict ourselves to the class of complex nilmanifolds and confirm the conjecture in that case. \end{abstract}

\maketitle


\markleft{Li and Zheng}
\markright{Complex nilmanifolds with constant $H$}

\section{Introduction and statement of result}

Given a Hermitian manifold $(M^n,g)$, denote by $R$ the curvature tensor of the Chern connection. The only possibly non-zero components of $R$ are $R_{X\overline{Y}Z\overline{W}}$ where $X$, $Y$, $Z$, $W$ are type $(1,0)$ tangent vectors. The holomorphic sectional curvature $H$ is defined by
$$ H(X) = R_{X\overline{X}X\overline{X}} / |X|^4 $$
for any non-zero type $(1,0)$ tangent vector $X$. When the metric $g$ is K\"ahler, it is well known that the values of $H$ determines the entire $R$, and complete K\"ahler manifolds with constant $H$ are exactly the complex space forms, namely, with universal covering space being ${\mathbb C}{\mathbb P}^n$, ${\mathbb C}^n$, or ${\mathbb C}{\mathbb H}^n$ equipped with (scaling of) the standard metric. When $g$ is non-K\"ahler, $R$ does not obey the usual symmetry conditions as in the K\"ahler case, and the values of $H$ do not determine the entire curvature tensor $R$. Nonetheless, when the manifold is compact, the following statement is still believed by many to be true:

\vspace{0.3cm}

\noindent {\bf Conjecture A.} {\em Let $(M^n,g)$ be a compact Hermitian manifold with  $H$  equal to a constant $c$. Then $g$ is K\"ahler if $c\neq 0$ and $g$ is Chern flat (namely, $R=0$) if $c=0$.}

\vspace{0.3cm}

Compact Chern flat manifolds were classified by Boothby \cite{Boothby} in 1958 as all the compact quotients of complex Lie groups, equipped with left invariant metrics.

 When $n=2$,  Balas and Gauduchon in 1985 \cite{BG} proved    the  $c\leq 0$ case of the conjecture (see also \cite{Gauduchon1} and \cite{Balas} for earlier work), and in 1996, Apostolov, Davidov and Muskarov \cite{ADM} solved the remaining $c>0$ case.

 For $n\geq 3$, the conjecture is still largely open, with only a few partial results known. For instance, Kai Tang in \cite{Tang} proved the conjecture under the additional assumption that the metric is Chern K\"ahler-like, meaning that the curvature tensor $R$ obeys all the K\"ahler symmetries. In their recent paper \cite{CCN}, Chen-Chen-Nie proved the conjecture for the case $c\leq 0$ under the additional assumption that $g$ is locally conformally K\"ahler. They also pointed out the necessity of the compactness assumption in the conjecture by explicit examples.

 In a recent work \cite{ZZ}, W. Zhou and the second named author proved that any compact Hermitian threefold with vanishing real bisectional curvature must be Chern flat. Real bisectional curvature is a curvature notion introduced by X. Yang and the second named author in \cite{YangZ}. It is equivalent to $H$ in strength when the metric is K\"ahler, but slightly stronger than $H$ in the non-K\"ahler case.

 For $n\geq 3$, Conjecture A seems to be a daunting task at this point in time, even for the $n=3$ case. Some people are actually hoping for counterexamples which would form an interesting class of non-K\"ahler manifolds shall they exist. One way to make the goal more realistic is by restricting to the following special case:

 \vspace{0.3cm}

\noindent {\bf Conjecture B.} {\em Let $(M^3,g)$ be a compact balanced Hermitian threefold with  $H=c$. Then $g$ is K\"ahler when  $c\neq 0$ and Chern flat when $c=0$.}

\vspace{0.3cm}

Recall that a Hermitian manifold $(M^n,g)$ is said to be balanced, if $d(\omega^{n-1})=0$, where $\omega$ is the K\"ahler form of $g$. The set of balanced threefolds already contains a lot of non-K\"ahler manifolds people are interested in, including all the  twistor spaces and many known examples of non-K\"ahler Calabi-Yau manifolds. In \cite[Proposition 2]{ZZ}, we observed a special unitary frame on such manifolds, which hopefully can be further exploited in approaching Conjecture B.

Another way of specialization is to restrict $(M^n,g)$ to all locally homogeneous Hermitian manifolds, namely, quotients of homogeneous complex manifolds equipped with invariant Hermitian metrics. A large and special subset of this class is the so-called  Lie-Hermitian manifolds, namely,

\vspace{0.3cm}

\noindent {\bf Definition C.}  A compact Hermitian manifold $(M^n,g)$ is called a {\em Lie-Hermitian manifold,} if its universal covering space is holomorphically isometric to $(G,J,g)$, where $G$ is an even-dimensional (connected and simply-connected) Lie group equipped with a left invariant complex structure $J$ and a compatible left invariant metric $g$.

\vspace{0.3cm}

In this case we will write $M = G/\Gamma$, where $\Gamma$ is the deck transformation group, consisting of holomorphic isometries of $(G,J,g)$.  Lie-Hermitian manifolds form a special and important class of locally homogeneous Hermitian manifolds. They are characterized by the existence of a flat Hermitian connection with parallel torsion (namely, a connection $\nabla'$ such that $\nabla'J=0$, $\nabla'g=0$, $\nabla'T'=0$ and $R'=0$, where $T'$ and $R'$ are the torsion and curvature of $\nabla'$, respectively). Since the difficulty in Conjecture A is largely caused by the non-symmetry of the curvature tensor $R$, the Lie-Hermitian manifolds, although locally homogeneous, already exhibit all the algebraic complexity and entanglement at one point. On the other hand, the local homogeneity makes the computation of curvature and torsion a lot more accessible, so it makes  sense to consider the following special case of Conjecture A:

\vspace{0.3cm}

\noindent {\bf Conjecture D.} {\em Let $(M^n,g)$ be a Lie-Hermitian manifold with  $H=c$. Then $g$ must be K\"ahler when $c\neq 0$ and $g$ must be Chern flat when $c=0$.}

\vspace{0.3cm}

When $G$ is nilpotent, the Lie-Hermitian manifold is called a {\em complex nilmanifold}. As a partial evidence, we prove Conjecture D for complex nilmanifolds:

\vspace{0.3cm}

\noindent {\bf Theorem E.} {\em Let $(M^n,g)$ be a complex nilmanifolds with  $H=c$. Then $c=0$, $R=0$, and $G$ is a complex Lie group.}

\vspace{0.3cm}

Note that when $G$ is a complex Lie group (and $J$ is the associated complex structure), it is always Chern flat, but the converse is not true. So in Theorem E the conclusion says a bit more than the Lie-Hermitian manifold being Chern flat. In general, two Lie-Hermitian manifolds can be holomorphically isometric but with the two Lie groups not isomorphic to each other. For instance, there are examples of non-abelian group $G$ where $(G,J,g)$ is K\"ahler and flat, thus holomorphically isometric to the complex Euclidean space ${\mathbb C}^n$. See \cite[Propsition 3.1]{BDF} or \cite[Appendix]{VYZ} for example. See also \cite{Milnor} for the characterization of Lie groups with flat left invariant metrics.

Beyond nilpotent groups, Conjecture D seems to be still quite challenging for us. Perhaps the next trial case would be Lie-Hermitian manifolds with abelian complex structures (see \cite{ABD} and the references therein for more details about this special type of solvmanifolds).

\vspace{0.3cm}

The article is organized as follows. In the next section, we will set up the notations, collect some known results from existing literature, and prove some preliminary lemmas. In the last section we will give a proof to Theorem E.

\vspace{0.3cm}

\section{Preliminaries }

Let $(M^n,g)$ be a Hermitian manifold and $R$ the curvature of its Chern connection. By definition, the holomorphic sectional curvature $H$ of $R$ is equal to a constant $c$ if and only if
$$ R_{X\overline{X}X\overline{X}} = c |X|^4  $$
for any type $(1,0)$ tangent vector $X$. Under a local unitary frame $e=\{ e_1, \ldots , e_n\}$, we have
\begin{equation}
H=c \ \Longleftrightarrow \ \widehat{R}_{i\overline{j}k\overline{\ell}}=\frac{c}{2}(\delta_{ij}\delta_{k\ell} + \delta_{i\ell}\delta_{kj}),
\end{equation}
where
\begin{equation}
\widehat{R}_{i\overline{j}k\overline{\ell}} = \frac{1}{4} \big( R_{i\overline{j}k\overline{\ell}} + R_{k\overline{j}i\overline{\ell}} +  R_{i\overline{\ell}k\overline{j}} +  R_{k\overline{\ell}i\overline{j}} \big) \label{Rhat}
\end{equation}
is the symmetrization of $R$. When $g$ is K\"ahler, the well-known K\"ahler symmetry says that $\widehat{R}=R$, so $H$ determines the entire $R$, but for general Hermitian metrics, $H$ can only determine $\widehat{R}$ but not $R$.

Next, let us recall the basic properties of Lie-Hermitian manifolds. Let $G$ be a connected, simply-connected, even-dimensional Lie group, and ${\mathfrak g}$ its Lie algebra. Left invariant complex structures and  compatible metrics  on $G$ correspond to almost complex structures $J$ and compatible inner products $g=\langle \,,\,\rangle$  on ${\mathfrak g}$, such that $J$ is integrable, namely,
\begin{equation*}
[x,y]- [Jx,Jy] +J [Jx, y] + J[x,Jy] =0
\end{equation*}
for any $x$, $y\in {\mathfrak g}$. Fix such a $({\mathfrak g}, J, \langle \,,\,\rangle)$. Extend $J$ and $\langle \,,\,\rangle$ linearly over ${\mathbb C}$ to the complexification ${\mathfrak g}^{\mathbb C}$. We have decomposition ${\mathfrak g}^{\mathbb C} = {\mathfrak g}' \oplus {\mathfrak g}''$ into $(1,0)$ and $(0,1)$ parts, where ${\mathfrak g}'' = \overline{{\mathfrak g}'}$, and
$${\mathfrak g}'= \{ x-\sqrt{-1}Jx \mid x\in {\mathfrak g}\} . $$
Suppose the real dimension of ${\mathfrak g}$ is $2n$. Let $e=\{ e_1, \ldots , e_n\}$ be a unitary basis of ${\mathfrak g}'$. Following the notations of \cite{VYZ} (see also \cite{YZ-Gflat} and \cite{ZhaoZ}), let us denote
\begin{equation}
C_{ik}^j = \langle [e_i,e_k] , \,\overline{e_j} \rangle , \ \ \ \ \ D_{ik}^j = \langle  [\overline{e_j}, e_k] , e_i \rangle
\end{equation}
for any $1\leq i,j,k\leq n$. Then by the integrability of $J$ we have
\begin{equation}
[e_i,e_k] = \sum_{j=1}^n C_{ik}^j e_j , \ \ \ \ \ [\overline{e_j}, e_k] = \sum_{i=1}^n \big( D_{ik}^j \overline{e_i} - \overline{D^k_{ij}} e_i  \big).
\end{equation}
We can extend each $e_i$  to left invariant vector fields on $G$, which we will still denote as $e_i$. So $e$ becomes a global unitary frame on $G$ as a Hermitian manifold. This will be our frame of choice from now on. Under this frame, the Chern connection $\nabla$ has expression
\begin{equation}
\nabla e_i  = \sum_{j=1}^n \theta_{ij}e_j = \sum_{j=1}^n  \sum_{k=1}^n \big( D^j_{ik}\varphi_k  -\overline{D^i_{jk}} \overline{\varphi_k} \big) \,e_j \label{theta}
\end{equation}
where $\varphi$ is the coframe of $(1,0)$-forms dual to $e$. Denote by $T$ the torsion tensor of the Chern connection $\nabla$, which is defined by $T(x,y)=\nabla_xy-\nabla_yx-[x,y]$ for any two vector fields. It is well-known that $T(X,\overline{Y})=0$ for any type $(1,0)$ tangent vectors $X$ and $Y$. Let us denote the components of $T$ under the frame $e$ by $T^j_{ik}$, namely, $\, T(e_i,e_k)   = \sum_{j=1}^n  \,T_{ik}^j e_j$. Note that our $T^j_{ik}$ here is twice of the same notation in \cite{YZ} or \cite{VYZ}. We have (see \cite{VYZ} for instance, and notice the factor $2$ difference)
\begin{equation}
T_{ik}^j = - C_{ik}^j - D_{ik}^j + D_{ki}^j.   \label{T}
\end{equation}
The covariant derivative of $T$ with respect to the Chern connection $\nabla$ are given by
\begin{eqnarray}
T^j_{ik,\,\ell} & = & \sum_{r=1}^n \big( -T^j_{rk} D^r_{i\ell} - T^j_{ir} D^r_{k\ell} + T^r_{ik} D^j_{r\ell } \big)   \label{Tderi} \\
T^j_{ik,\,\overline{\ell}} & = & \sum_{r=1}^n \big( T^j_{rk} \overline{D^i_{r\ell}} + T^j_{ir} \overline{D^k_{r\ell} } - T^r_{ik} \overline{ D^r_{j\ell } }\big) \label{Tderibar}
\end{eqnarray}
for any $i$, $j$, $k$, $\ell$, where the indices after comma stand for covariant derivatives with respect to $\nabla$. By \cite[Lemma 2.1]{VYZ}, the constants $C$ and $D$ also satisfy the following conditions given by the Jacobi identity:
\begin{eqnarray}
& & \sum_{r=1}^n \big( C^r_{ij}C^{\ell}_{rk} + C^r_{jk}C^{\ell}_{ri} + C^r_{ki}C^{\ell}_{rj} \big) \ = \ 0  \label{CC} \\
& & \sum_{r=1}^n \big( C^r_{ik}D^{\ell}_{jr} + D^r_{ji}D^{\ell}_{rk} - D^r_{jk}D^{\ell}_{ri} \big) \ = \ 0  \label{CD} \\
& & \sum_{r=1}^n \big( C^r_{ik}\overline{D^{r}_{j\ell }} - C^j_{rk}\overline{D^{i}_{r\ell }} + C^j_{ri}\overline{D^{k}_{r\ell }}  - D^{\ell}_{ri}\overline{D^{k}_{jr }} + D^{\ell}_{rk} \overline{ D^{i}_{jr }}  \big) \ = \ 0  \label{CDbar}
\end{eqnarray}
for any $i$, $j$, $k$, $\ell$. For any Lie-Hermitian manifold, the components of the curvature tensor $R$ of $\nabla$ take a particularly simple form:


\begin{lemma}
Under any unitary basis $e$ of ${\mathfrak g}$, the components of Chern curvature tensor $R$ of a Lie-Hermitian manifold are given by
\begin{equation}
R_{i\overline{j}k\overline{\ell}} = \sum_{r=1}^n \big( D^r_{ki} \overline{ D^r_{\ell j} } - D^{\ell}_{ri} \overline{ D^k_{rj } }  - D^j_{ri} \overline{ D^k_{\ell r} } -  \overline{ D^i_{rj} } D^{\ell}_{kr}  \big) \label{R}
\end{equation}
for any $1\leq i,j,k,\ell \leq n$.
\end{lemma}

\begin{proof}
By (\ref{theta}), we have $\nabla_{e_k}e_i = \sum_r D_{ik}^r e_r$, $\ \nabla_{\overline{e_j}}e_i = - \sum_r \overline{ D_{rj}^i } e_r$. So we compute
\begin{eqnarray*}
R_{i\overline{j}k\overline{\ell}}  & = & \langle   \nabla_{e_i} \nabla_{\overline{e_j}}e_k - \nabla_{\overline{e_j}} \nabla_{e_i} e_k - \nabla_{[e_i, \overline{e_j} ]} e_k     , \, \overline{e_{\ell}} \rangle  \\
& = & \langle   \nabla_{e_i} ( - \overline{D^k_{rj}} e_r)  - \nabla_{\overline{e_j}} (D^r_{ki}e_r ) - \nabla_{(-D^j_{ri}\overline{e_r} + \overline{D^i_{rj}} e_r ) } e_k     , \, \overline{e_{\ell}} \rangle  \\
& = & - \overline{D^k_{rj}} D^{\ell}_{ri} +  D^r_{ki} \overline{D^r_{\ell j}}    - D^j_{ri} \overline{D^k_{\ell r} } - \overline{D^i_{rj}} D^{\ell}_{kr}
\end{eqnarray*}
where the index $r$ is summed up from $1$ to $n$. This proves the lemma.
\end{proof}

In particular, we have
\begin{equation}
R_{i\overline{i}i\overline{i}} = \sum_{r=1}^n \big( |D^r_{ii}|^2 - |D^i_{ri}|^2 - 2 {\mathfrak R}\mbox{e} \{  D^i_{ri} \overline{D^i_{ir}} \} \big) \label{HD}
\end{equation}
for any $1\leq i\leq n$, and by (\ref{Rhat}) and (\ref{R}), we immediately obtain the expression for $\widehat{R}$, and in particular, we have

\begin{lemma}
Under any unitary basis $e$ of ${\mathfrak g}$, it holds that
\begin{equation}
\widehat{R}_{i\overline{i}k\overline{k}} = \sum_{r=1}^n \big( |D^r_{ki}+D^r_{ik}|^2 - |D^k_{ri}|^2 -|D^i_{rk}|^2 - 2 {\mathfrak R}\mbox{e} \{ D^k_{rk} \overline{D^i_{ri} } +  D^i_{ri} \overline{D^k_{kr} } + D^k_{rk} \overline{D^i_{ir} } + D^i_{rk} \overline{D^i_{kr} } + D^k_{ri} \overline{D^k_{ir} }\} \big) \label{RhatD}
\end{equation}
for any $1\leq i,k \leq n$.
\end{lemma}

\vspace{0.3cm}

\section{Proof of Theorem E }

In this section, we give a proof of Theorem E stated in the first section. Let $G$ be a nilpotent Lie group equipped with a left invariant  integrable complex structure $J$ and compatible metric $g=\langle \, ,\, \rangle$.

First let us recall a famous result of Salamon \cite[Theorem 1.3]{Salamon} on nilpotent Lie groups with complex structure:

\begin{theorem} [Salamon] Let $G$ be a nilpotent Lie group of dimension $2n$ equipped with a left invariant complex structure. Then there exists a coframe $\varphi =\{ \varphi_1, \ldots , \varphi_n\}$ of left invariant $(1,0)$-forms on $G$ such that $$ d\varphi_1 =0, \ \ \ d\varphi_i = {\mathcal I} \{\varphi_1, \ldots , \varphi_{i-1}\} , \ \ \ \forall \ 2\leq i\leq n, $$
where ${\mathcal I}$ stands for the ideal in  exterior algebra of the complexified cotangent bundle generated by those $(1,0)$-forms.
\end{theorem}

Clearly, one can also assume that $\varphi$ is unitary. In terms of the structure constants $C$ and $D$, this means that
\begin{equation}
C^j_{ik}=0  \ \ \ \mbox{unless} \ \ j>i \ \mbox{or} \ j>k; \ \ \ \ \ D^j_{ik}=0  \ \ \ \mbox{unless} \ \ i>j.  \label{Salamon}
\end{equation}
In particular, we always have
\begin{equation}
D^n_{ik} = 0 \ \ \ \forall \  1\leq i,k\leq n.  \label{Duppern}
\end{equation}

Now let us assume that $G$ has constant holomorphic sectional curvature $H=c$. Let $e$ be the unitary basis of ${\mathfrak g}$ dual to the unitary Salamon coframe $\varphi$ above. Since $D^i_{ir}=0$, so by (\ref{HD}) we  have
$$ R_{i\overline{i}i\overline{i}} = \sum_{r=1}^n  |D^r_{ii}|^2 - \sum_{r=1}^n |D^i_{ri}|^2 = c $$
for each $1\leq i \leq n$. In particular,
$$c=R_{1\overline{1}1\overline{1}} = - \sum_r |D^1_{r1}|^2 \leq 0, \ \ \ \ \ \mbox{ and} \ \ \ \  \ c= R_{n\overline{n}n\overline{n}} =  \sum_r |D^r_{nn}|^2 \geq 0. $$
Therefore we must have $c=0$, $ D^1_{r 1}=0$, and
\begin{equation}
D^r_{nn} = 0, \ \ \ \forall \ 1\leq r \leq n.\label{Dlowernn}
\end{equation}
For any $1\leq i<k\leq n$, we have $\widehat{R}_{i\overline{i}k \overline{k}} = \frac{c}{2}(1+\delta_{ik}) = 0$, so by (\ref{RhatD}), we get
\begin{equation}
\sum_{r=1}^n  |D^r_{ki}+D^r_{ik}|^2 =   \sum_{r=1}^n \big(  |D^k_{ri}|^2 + |D^i_{rk}|^2 +  2 {\mathfrak R}\mbox{e} \{ D^k_{rk} \overline{D^i_{ri} } +   D^i_{rk} \overline{D^i_{kr} }    \}  \big), \ \   \ \ \ \forall \ i<k.  \label{formula9}
\end{equation}
We want to deduce $D=0$ from this. Let $k=n$ in (\ref{formula9}) and use (\ref{Duppern}), we find
\begin{equation}
\sum_{r=1}^n  |D^r_{ni}+D^r_{in}|^2 =   \sum_{r=1}^n \big(  |D^i_{rn}|^2 +  2 {\mathfrak R}\mbox{e} \{   D^i_{rn} \overline{D^i_{nr} }    \}  \big), \ \   \ \ \ \forall \ i<n.  \label{formula10}
\end{equation}
Now let $i=n-1$ in (\ref{formula10}). By the fact $D^{\ast}_{nn}=0$ from (\ref{Dlowernn}), we know that the right hand side of (\ref{formula10}) with $i=n-1$ must be zero, so we get
\begin{equation}
D^r_{n,n\!-\!1}+D^r_{n\!-\!1,n} =   0, \ \ \ \forall \ r  \label{formula11}
\end{equation}
In particular, $D^{n\!-\!1}_{n,n\!-\!1}=0$. Next, let $i=n-2$ in (\ref{formula10}), and use (\ref{formula11}), the right hand side becomes
$$   |D^{n\!-\!2}_{n\!-\!1,n}|^2 + 2 {\mathfrak R}\mbox{e} \{ D^{n\!-\!2}_{n\!-\!1,n} \overline{D^{n\!-\!2}_{n,n\!-\!1} } \}  = |D^{n\!-\!2}_{n\!-\!1,n}|^2 - 2|D^{n\!-\!2}_{n\!-\!1,n}|^2  = -  |D^{n\!-\!2}_{n\!-\!1,n}|^2 , $$
while the left hand is
\begin{equation*}
\sum_{r=1}^{n\!-\!3} |D^r_{n,n\!-\!2}   +D^r_{n\!-\!2,n}  |^2 + \sum_{r=n\!-\!2}^{n\!-\!1} |D^r_{n,n\!-\!2} |^2.
\end{equation*}
So we conclude that
\begin{equation}
D^{n\!-\!1}_{n,n\!-\!2}= D^{n\!-\!2}_{n,n\!-\!2} = D^{n\!-\!2}_{n\!-\!1,n} =0, \ \ \ \ \ D^r_{n,n\!-\!2}+ D^r_{n\!-\!2,n} =0 \ \ \forall \ r. \label{formula12}
\end{equation}
We claim that formula (\ref{formula10}) implies that
\begin{equation}
D^{i}_{rn}= 0, \ \ \ D^{r}_{ni} + D^{r}_{in} =0, \ \ \   \forall \ r,\   \ \forall \ i. \label{claim}
\end{equation}
In the above discussion, we already see that the claim holds for $i=n$, $i=n-1$, and $i=n-2$. Suppose the claim (\ref{claim}) is true for $i=n, n-1, n-2, \ldots , j+1$. Then for $i=j$, in the right hand side of (\ref{formula10}) we have
$$ D^i_{rn}\overline{D^i_{nr}} = - |D^i_{rn}|^2, \ \ \ \forall \ r>i.$$
So formula (\ref{formula10}) becomes
$$ \sum_{r=1}^n  |D^r_{ni}+D^r_{in}|^2 =   - \sum_{r=1}^n   |D^i_{rn}|^2  $$
and the claim holds for $i=j$ as well. By induction, we know that the claims for all $i$. In particular, we have
$$ D^n_{\ast \ast} = D^{\ast}_{n\ast} = D^{\ast}_{\ast n} = 0.$$
So whenever the index $n$ appears, $D$ will be zero. Repeat the argument for $n-1$ and so on, we eventually get $D^j_{ik}=0$ for all $i,j,k$. This means that $G$ is a complex Lie group, and is automatically Chern flat. This completes the proof of Theorem E. \qed

\vspace{0.3cm}

As a final remark, as we commented at the end of the first section, the conjecture still seems quite challenging for all Lie-Hermitian manifolds. So the next special case might be all Lie-Hermitian manifolds with abelian complex structure, namely, when $C^j_{ik}=0$ for any $i$, $j$, $k$. In this case, the Chern curvature tensor $R$ is still given by the quadratic expression in $D$ as (\ref{R}), but $D$ satisfy the restriction
$$ \sum_r   D^r_{ji}D^{\ell}_{rk} = \sum_r  D^r_{jk}D^{\ell}_{ri} ; \ \ \ \ \
\sum_r  D^{\ell}_{ri}\overline{D^{k}_{jr }} = \sum_r D^{\ell}_{rk} \overline{ D^{i}_{jr }}   $$
for any $i\neq k$ and any $j$, $\ell$. Hopefully one can exploit these commutativity conditions sufficiently to obtain a proof  of Conjecture D in this special case.

\vspace{0.5cm}

\noindent\textbf{Acknowledgments.} {The second named author would like to thank mathematicians  Haojie Chen, Xiaolan Nie, Kai Tang, Bo Yang, Xiaokui Yang, and Quanting Zhao for their interests and/or helpful discussions.}


\end{document}